\documentclass[12pt, a4paper]{article}
\usepackage{amsmath, amssymb, amsthm}
\usepackage{graphics}
\usepackage{graphicx}
\usepackage{epsfig}
\usepackage{dblfloatfix}
\usepackage{float}
\usepackage{placeins}
\usepackage{flafter}
\usepackage[usenames]{color}

\usepackage{epstopdf}
\usepackage{url}
\usepackage{hyperref}
\date{}
\usepackage{fullpage}
\usepackage{setspace}

\newtheorem{ex}{Example}[section]
\newtheorem{df}{Definition}[section]
\newtheorem{pro}{Proposition}[section]
\newtheorem{thm}{Theorem}[section]
\newtheorem{cor}{Corollary}[section]

\makeatletter

\renewcommand\section{\@startsection {section}{1}{\z@}
{-30pt \@plus -1ex \@minus -.2ex} {2.3ex \@plus.2ex}
{\normalfont\normalsize\bfseries}}

\renewcommand\subsection{\@startsection{subsection}{2}{\z@}
{-3.25ex\@plus -1ex \@minus -.2ex} {1.5ex \@plus .2ex}
{\normalfont\normalsize\bfseries}}

\renewcommand{\@seccntformat}[1]{\csname the#1\endcsname. }

\makeatother
\title{\bf ON ALMOST ARMENDARIZ RINGS}
\author{  Sushma Singh \\
Department of Mathematics, \\
Indian Institute of Technology Patna, Bihar, India 801103 \\
E-mail: sushmasingh@iitp.ac.in \\\\
 Om Prakash \\
Department of Mathematics, \\
Indian Institute of Technology Patna, Bihar, India 801103 \\
E-mail: om@iitp.ac.in}

\begin{document}

\maketitle

\begin{abstract} In this paper, we introduce the notion of an almost Armendariz ring which is a generalization of an Armendariz ring and discuss some of its properties. In 2006, Liu and Zhao introduced the concept of weak Armendariz ring. It has been observed that, every almost Armendariz ring is weak Armendariz but converse need not be true. We prove that a ring $R$ is an almost Armendariz if and only if $R[x]$ is an almost Armendariz. It is also shown that if $R/I$ is an almost Armendariz ring and $I$ is semicommutative ideal, then $R$ is an almost Armendariz ring. Moreover, the class of minimal noncommutative almost Armendariz rings is completely determined, up to isomorphism (minimal means having smallest cardinality).
\end{abstract}


\section{INTRODUCTION}
Throughout this paper, $R$ denotes an associative ring with identity $1\neq 0$ and $R[x]$ is the usual polynomial ring over $R$ in indeterminate $x$. For a polynomial $f(x) \in R[x]$, $C_{f(x)}$ denotes the set of all coefficients of $f(x)$. $M_{n}(R)$ and $U_{n}(R)$ denote the $n\times n$ full matrix ring and upper triangular matrix ring over $R$ respectively. $D_{n}(R)$ is the ring of $n\times n$ upper triangular matrices over $R$ whose diagonal entries are equal. We use $e_{ij}$ for the matrix with $(i, j)^{th}$ entry $1$ and $0$ otherwise.\\
For a ring $R$, $N(R)$ denotes the set of all nilpotent elements of $R$.
We also know that an element $a$ of a ring $R$ is strongly nilpotent if every sequence $a_{1}, a_{2}, a_{3}\ldots$ such that $a_{1} = a$ and $a_{n+1}\in a_{n}Ra_{n}$ (for all $n \in \mathbb{Z_{+}}$) is eventually zero, i.e. there exist a positive integer $n$ such that $a_{n} = 0$. Recall that the lower nil radical (prime radical) of a ring $R$ is defined by the intersection of all prime ideals of $R$ and it is denoted by $N_{*}(R)$. It is precisely the collection of all strongly nilpotent elements of $R$, i.e., $N_{*}(R) = \{x\in R : RxR$ is nilpotent$\}$. $N^{*}(R)$ denotes the upper nil radical (i.e sum of nil ideals of $R$). It is known that $N_{*}(R) \subseteq N^{*}(R) \subseteq N(R)$.
Due to Birkenmeier et al.\cite{G}, a ring $R$ (without identity) is said to be 2-primal if $N_{*}(R) = N(R)$ and by Mark \cite{GG} , $R$ is said to be NI ring if $N^{*}(R) = N(R)$.\\
A ring $R$ is said to be reduced if it has no nonzero nilpotent elements. In 1974, Armendariz  had pointed out that a reduced ring is satisfying the condition of Lemma 1 of \cite{E}. In 1997, term Armendariz ring was coined by Rege and Chhawchharia in \cite{M}.

\begin{df}
A ring $R$ is said to be Armendariz if for two polynomials $f(x)$ and $g(x)\in R[x]$, $f(x)g(x) = 0$ implies $ab = 0$ for each $a \in C_{f(x)}$ and $b \in C_{g(x)}$.
\end{df}

For more details the reader is referred to [\cite{D, R, E, J, N,TT, Z}]. In 2006, Liu and Zhao  generalized the concept of Armendariz ring and introduced the weak Armendariz ring in \cite{Z}.

\begin{df}{[Liu and Zhao \cite{Z}]}
A ring $R$ is said to be weak Armendariz if for two polynomials $f(x)$ and $g(x)\in R[x]$ such that $f(x)g(x) = 0$ implies $ab \in N(R)$ for each $a \in C_{f(x)}$ and $b \in C_{g(x)}$.
\end{df}
A ring $R$ is said to be semicommutative if $ab = 0$ implies $aRb = 0,$ for each $a, b \in R$. Semicommutative ring was introduced by Shin [\cite{GGG}, Lemma 1.2].
Liu and Zhao proved that if $R$ is semicommutative ring, then polynomial ring $R[x]$ over weak Armendariz ring is weak Armendariz ring. In 2008, the concept of nil Armendariz ring, which is the generalization of weak Armendariz ring, introduced by Antoine \cite{R}. A ring $R$ is said to be nil-Armendariz ring if whenever two polynomials $f(x) = \sum_{i=0}^{m}a_{i}x^{i}$ and $g(x) = \sum_{j=0}^{n}b_{j}x^{j}\in R[x]$ such that $f(x)g(x) \in N(R)[x]$ implies $a_{i}b_{j}\in N(R)$, for each $i$, $j$, where $0 \leq i \leq m$ and $0\leq j \leq n$. He proved that a ring $R$ is nil-Armendariz if and only if $R/I$ is nil-Armendariz ring, where $I$ is a nil ideal of $R$.\\

In 2013, Han et al. \cite{J} introduced the concept of Armendariz-over-prime-radical ($APR$) and defined that \emph{a ring $R$ is said to be an APR if $f(x)g(x) \in N_{*}(R)[x]$ implies $ab \in N_{*}(R)$ for each $a \in C_{f(x)}$ and $b \in C_{g(x)}$}. Clearly, $APR$ rings are nil-Armendariz ring but converse is not true.

It is known that $(0)\subseteq N_{*}(R) \subseteq N(R)$. Therefore, motivated by above we introduce the notion of an almost Armendariz ring involving $N_{*}(R)$. Some of results on lower nil radical can be viewed in [\cite{KK, T}].

\section{ALMOST ARMENDARIZ RINGS}

\begin{df}
A ring $R$ is said to be an almost Armendariz ring if for two polynomials $f(x)$ and $g(x)\in R[x]$ such that $f(x)g(x) = 0$ implies $ab \in N_{*}(R)$ for each $a \in C_{f(x)}$ and $b \in C_{g(x)}$.
\end{df}
Clearly, a subring of an almost Armendariz ring is an almost Armendariz. Every Armendariz ring is an almost Armendariz ring and every almost Armendariz ring is weak Armendariz, but  converse need not be true. If $R$  is commutative, then every weak Armendariz ring is an almost Armendariz. Also, if $R$ is a nil ring, then both are equivalent.\\

\begin{ex}
Let $R = \mathbb{Z}_{5}[x, y]/<x^{3}, x^{2}y^{2}, y^{3}>$, where $\mathbb{Z}_{5}$ is Galois field of order 5, $\mathbb{Z}_{5}[x, y]$ be polynomial ring with commuting indeterminate $x, y$ and $<x^{3}, x^{2}y^{2}, y^{3}>$ be the ideal of $\mathbb{Z}_{5}[x, y]$ generated by $x^{3}, x^{2}y^{2}, y^{3}$. Here, $R$ is not an Armendariz ring because if we take two polynomial $f(t) = (x+yt)$, $g(t) = (3x^{2}+2xyt+3y^{2}t^{2})$, then $f(t)g(t) = 0$ but $x.2xy \neq 0$. We can easily see that $R$ is an almost Armendariz ring.
\end{ex}

\begin{ex} Let $R$ be a reduced ring. Then the ring $D_{n}(R)$ is not an Armendariz by Kim and Lee [\cite{N}, Example 3] when $n \geq4$. But $D_{n}(R)$ is an almost Armendariz ring by Proposition (3.3).
\end{ex}

\begin{ex}
We refer the construction of ring in [\cite{S}, Example 1.2]. Let $S$ be a reduced ring, $n$ a positive integer and $R_{n} = U_{2^{n}}(S)$. Then each $R_{n}$ is NI ring by Proposition 4.1(1) of \cite{S}. Define a map $\sigma : R_{n}\rightarrow R_{n+1}$ by $A\mapsto \left(
                                                                                                            \begin{array}{cc}
                                                                                                              A & 0 \\
                                                                                                              0 & A \\
                                                                                                            \end{array}
                                                                                                          \right).$
Then $R_{n}$ can be considered as subring of $R_{n+1}$ via $\sigma$ $(i.e.,~A = \sigma (A) ~for~ A \in R_{n} )$. Notice that $D = \{R_{n}, \sigma_{nm}\}$, with $\sigma_{nm} = \sigma^{m-n}$ whenever $n \leq m$, is a direct system over $I = \{1, 2, \ldots\}$. Set $R = \underrightarrow {Lim}~ R_{n}$ be the direct limit of $D$. Then $R = \bigcup_{n = 1}^{\infty}R_{n}$, and $R$ is $NI$ by Proposition 1.1 of \cite{S}. Since $NI$ rings are nil Armendariz rings and every nil Armendariz ring is weak Armendariz ring. Therefore, $R$ is weak Armendariz ring. Also by Theorem 2.2(1) of \cite{YYY}, $R$ is a semiprime ring, hence $N_{*}(R) = 0$. Here, $R$ is not $APR$, by Lemma 1.1(7) of \cite{J}, since $N^{*}(R) = \{m = (m_{ij})\in R ~|~ m_{ii} = 0 ~for ~all ~i\} \neq 0$. Since $N_{*}(R) = 0$, therefore $R$ is not almost Armendariz ring. In fact, if we take the polynomials $f(x) = e_{11} + e_{12}x$, $g(x) = e_{22}-e_{12}x$ in $R[x]$, then $f(x)g(x) = 0$ but $e_{11}e_{12} \neq 0$ and $e_{12}(-e_{22}) = -e_{12} \neq 0.$
\end{ex}

\begin{pro} Let $R$ be an almost Armendariz ring. Then
\begin{itemize}
\item [(1)] For $a, b, c \in R$, $ab = 0$ and $c^{n} = 0$ imply that $acb \in N_{*}(R)$, so $acb \in N(R)$.
\item [(2)] For $a, b \in R$, $a^{2} = 0$ and $b^{n} = 0$ imply that $ab \in N(R)$.
\item [(3)]For $a, b \in R$, $a^{2} = 0$ and $b^{n} = 0$ imply that $(a+b) \in N(R)$.
\end{itemize}
\end{pro}
\begin{proof}
\begin{itemize}
\item[(1)] Let $f(x) = a(1-cx)$, $g(x) = (1+cx+c^{2}x^{2}+\ldots+c^{n-1}x^{n-1})b$. Then $f(x)g(x) = 0$. Since $R$ is an almost Armendariz ring, therefore $acb \in N_{*}(R)$ and hence $acb \in N(R)$.

\item[(2)] Let $f(x) = a(1-bx)$, $g(x) = (1+bx+b^{2}x^{2}+\ldots+b^{n-1}x^{n-1})a$. Then $f(x)g(x) = 0$. Since $R$ is an almost Armendariz ring, therefore $aba \in N_{*}(R)$ this implies $abab \in N_{*}(R)$. Since $N_{*}(R)\subseteq N(R)$, therefore $(ab)^{2} \in N(R)$ and hence $ab \in N(R)$.

\item[(3)] \textbf{Case (i)} If $a^{2} = 0$ and $b^{2} = 0$, and we take $f(x) = a(1-bx)$, $g(x) = (1+bx)a$, then $f(x)g(x) = 0$. Since $R$ is an almost Armendariz ring, therefore $aba \in N_{*}(R)$.\newline
Now, $$(a+b)^{2} = a^{2}+ab+ba+b^{2} = ab+ba$$ and $$(a+b)^{3} = (a+b)(ab+ba) = a^{2}b+aba+bab+b^{2}a$$ and $$(a+b)^{4} = (a+b)(aba+bab) = a^{2}ba+abab+baba+b^{2}ab = abab+baba \in N_{*}(R).$$  That is $(a+b)^{4} \in N_{*}(R)$, since $N_{*}(R)\subseteq N(R)$, therefore $(a+b)^{4} \in N(R)$ and hence $(a+b) \in N(R)$.\newline \newline

\textbf{Case (ii)} If $a^{2} = 0$, $b^{3} = 0$, and $f(x) = a(1-bx)$, $g(x) = (1+bx+b^{2}x^{2})a$, then $f(x)g(x) = 0$. Since $R$ is an almost Armendariz ring, therefore $aba, ab^{2}a \in N_{*}(R)$.\\
Now, $$(a+b)^{2} = a^{2}+ab+ba+b^{2} = ab+ba+b^{2},$$ $$(a+b)^{3} = (a+b)(ab+ba+b^{2}) = a^{2}b+aba+ab^{2}+bab+b^{2}a+b^{3}$$ \\ $$= aba+ab^{2}+bab+b^{2}a+b^{3} = aba+ab^{2}+bab+b^{2}a.$$ By similar argument, we get $(a+b)^{6} = ababab+abab^{2}a+ab^{2}aba+ab^{2}ab^{2}+bababa+babab^{2}+bab^{2}ab+b^{2}abab+b^{2}ab^{2}a \in N_{*}(R)$, therefore $(a+b)^{6} \in N(R)$ and hence $(a+b) \in N(R)$. Similarly,\newline \newline

\textbf{Case (iii)} If $a^{2} = 0$, $b^{4} = 0$, we get $(a+b)^{8} \in N_{*}(R)$ and hence $(a+b) \in N(R)$. \\

\textbf{Case (iv)} $a^{2} = 0$, $b^{5} = 0$, we have $(a+b)^{10} \in N_{*}(R)$. Therefore, $(a+b) \in N(R)$. \\

Hence, in continuation, we conclude that if $a^{2} = 0$, $b^{n} = 0$, we get $(a+b)^{2n} \in N_{*}(R)$, since  $N_{*}(R)\subseteq N(R)$, therefore $(a+b)^{2n} \in N(R)$. Thus, $(a+b) \in N(R)$.
\end{itemize}
\end{proof}

\begin{pro} In an almost Armendariz ring $R$, the following are equivalent:
\begin{itemize}
\item[(1)] For $a, b \in R$, $a^{3} = 0$ and $b^{3} = 0$ imply $a+b \in N(R)$.
\item[(2)] For $a, b \in R$, $a^{3} = 0$ and $b^{3} = 0$ imply $ab \in N(R)$.
\end{itemize}
\end{pro}
\begin{proof} In this case we have four possibilities for polynomials $f(x)$ and $g(x)$ in $R[x]$.\\
(i) If $f(x) = a(1-bx)$, $g(x) = (1+bx+b^{2}x^{2})a^{2}$, then $f(x)g(x) = 0$. Since $R$ is an almost Armendariz ring, therefore $aba^{2}, ab^{2}a^{2} \in N_{*}(R)$.\\
(ii) If $f(x) = a^{2}(1-bx)$, $g(x) = (1+bx+b^{2}x^{2})a$, then $f(x)g(x) = 0$ and this implies $a^{2}ba, a^{2}b^{2}a \in N_{*}(R)$.\\
If we interchange $a$ and $b$, then we will get the other two and here $b^{2}ab, b^{2}a^{2}b, bab^{2}, ba^{2}b^{2} \in N_{*}(R)$.\\
\textbf{$(1)\Rightarrow(2)$:} To prove $ab\in N(R)$, first we see about nilpotency of $a+b$. Clearly, all terms of the expansion of
 $(a+b)^{5}$ contained in $N_{*}(R)$, except $ababa$ and $babab$. Since $(a+b)$ is nilpotent, so there exist a positive integer $n$ such that $(a+b)^{n} = 0$. If $n$ is even, then $0 = \alpha +(ab)^{n/2}+ (ba)^{n/2}$, where $\alpha \in N_{*}(R)$ and therefore $(ab)^{n/2}+ (ba)^{n/2} \in N_{*}(R)$. Now multiplying by $ab$ from right, we get $(ab)^{n/2}(ab)+(ba)^{n/2}(ab) \in N_{*}(R)$, but $(ba)^{n/2}(ab) \in N_{*}(R)$. Therefore, $(ab)^{n/2+1} \in N_{*}(R)$ and hence $ab \in N(R)$.\\
If $n$ is odd, then $0 = (a+b)^{n} = (a+b)^{n+1} = \beta +(ab)^{(n+1)/2}+ (ba)^{(n+1)/2}$, where $\beta \in N_{*}(R)$ and hence $(ab)^{(n+1)/2}+ (ba)^{(n+1)/2} \in N_{*}(R)$. Now multiplying by $ab$ from right, we get $(ab)^{(n+3)/2}\in N_{*}(R)$. Thus $ab \in N(R)$.\\

\textbf{$(2)\Rightarrow(1)$:} Since all terms of $(a+b)^{5}$ are in $N_{*}(R)$ except $ababa$ and $babab$. Also $ab$ is nilpotent and so is $ba$. Therefore by increasing power of $(a+b)$ till $ab$ and $ba$ reached to their nilpotency. In this way, we get a positive integer $m$ such that $(a+b)^{m} \in N_{*}(R)$. Thus $(a+b) \in N(R)$.
\end{proof}

A ring $R$ is abelian if every idempotent element is central. In 1998, Anderson and Camillo proved that Armendariz rings are abelian. But, it is seen that an almost Armendariz ring need not be abelian. In this regards we have the following:
\begin{ex} By following Proposition (3.2), $R$ is an almost Armendariz ring if and only if $U_{n}(R)$ is an almost Armendariz ring. But, $e_{11}e_{1n} \neq e_{1n}e_{11}$, where $e_{11}^{2} = e_{11}$ and $e_{11}, e_{1n} \in U_{n}(R)$. Hence, an almost Armendariz ring is not an abelian.
\end{ex}
\begin{pro} If $R$ is an almost Armendariz ring, then $R/N_{*}(R)$ is an abelian.
\end{pro}
\begin{proof} Let $e\in R$ be an idempotent and consider $a= e$, $b = (1-e)$. Then $a$ and $b$ are idempotents in $R$ and $c = er(1-e) \in R$ for each $r \in R$. Also,  $e + N_{*}(R), (1-e) + N_{*}(R)$ are idempotent elements of $R/N_{*}(R)$ and $er(1-e)+N_{*}(R) \in R/N_{*}(R)$. Now, by Propsition [2.5 (1)], $eer(1-e)(1-e) = er-ere \in N_{*}(R)$. Again, if we take $a^{'} = (1-e)$, $b^{'} = e$ and $c^{'} = (1-e)re$ are in $R$, then $(1-e)(1-e)ree = (1-e)re = re-ere \in N_{*}(R)$. Therefore, $(er-ere)-(re-ere) = er-re \in N_{*}(R)$. Thus, $R/N_{*}(R)$ is abelian.
\end{proof}

\begin{pro} Let $R$ be a ring and $e$ an idempotent element of $R$. If $e$ is central in $R$, then the following are equivalent:
\begin{itemize}
\item[(1)] $R$ is an almost Armendariz.
\item [(2)]$eR$ and $(1-e)R$ are almost Armendariz ring.
\end{itemize}
\end{pro}
\begin{proof}${(1)}\Rightarrow {(2).}$ It is obvious, since $N_{*}(eR) = eN_{*}(R)$ and $N_{*}((1-e)R) = (1-e)N_{*}(R)$.

${(2)}\Rightarrow {(1).}$ Let $f(x) = \sum_{i=0}^{m}a_{i}x^{i}$ and $g(x) = \sum_{j=0}^{n}b_{j}x^{j} \in R[x]$ be such that $f(x)g(x) = 0$. Then $(ef)(eg) = 0$ and $(1-e)f(1-e)g = 0$. Since $eR$ is an almost Armendariz ring, therefore $ea_{i}b_{j} \in N_{*}(R)$. Similarly, $(1-e)a_{i}b_{j} \in N_{*}(R)$, since $(1-e)R$ is also an almost Armendariz ring. Therefore $a_{i}b_{j} \in N_{*}(R)$ for each $i, j$, where $0 \leq i \leq m$ and $0 \leq j \leq n$. Thus $R$ is an almost Armendariz ring.
\end{proof}
A ring $R$ is said to be weakly semicommutative if for any $a, b \in R$, $ab = 0$, then $arb \in N(R)$ for each $r \in R$.\\
Here, we observe that an almost Armendariz ring need not be weakly semicommutative ring and vice-versa. \\
In Example $(2.3)$, it is shown that $R$ is an NI ring. Let $ab = 0$. Then $ba \in N(R)$ and hence $baR \subseteq N(R)$. This implies $aRb \subseteq N(R)$. Therefore, $R$ is weakly semicommutative ring.

\begin{ex} Let $K$ be a field and $R = K[a, b]/<a^{2}>$. By Example (4.8) of \cite{R}, $R$ is an Armendariz ring. Therefore, $R$ is an almost Armendariz ring. But it is not a weakly semicommutative, because, $(ba)a = 0$ but $(ba)b(a)$ is not a nilpotent element of $R$.
\end{ex}
A ring $R$ is locally finite if every finite subset in it generates a finite semigroup mutiplicatively.
\begin{pro} Let $R$ be locally finite abelian ring. If $R$ is an almost Armendariz ring, then $R$ is a weakly semicommutative ring.
\end{pro}
\begin{proof} Let $ab = 0$. For any $r \in R$, since $R$ is locally finite, therefore there exist positive integers $m, n$ such that $r^{m} = r^{m+n}$. Then inductively, we have $r^{m} = r^{m}r^{n} = r^{m}r^{2n} = \ldots = r^{m}r^{mn} = r^{m(n+1)}$. Put $t = n+1$, then $r^{m} = (r^{m})^{t}$. Observe that, $r^{(t-1)m} = r^{(t-2)m}r^{m} = r^{(t-2)m}(r^{m})^{t} = r^{2(t-1)m} = (r^{(t-1)m})^{2}$, hence $r^{(t-1)m}$ is an idempotent. Since $R$ is an abelian, therefore, $ar^{(t-1)m}b = 0$. Now, by Proposition $(2.1)$, $arb \in N(R)$. Thus, $R$ is a weakly semicommutative ring.
\end{proof}


\begin{pro} Every 2-primal ring is an almost Armendariz ring.
\end{pro}
\begin{proof} Let $f(x) = \sum_{i=0}^{m}a_{i}x^{i}$ and $g(x) = \sum_{j=0}^{n}b_{j}x^{j}\in R[x]$ such that $f(x)g(x) = 0$. Then \\

$\begin{array}{ll}
a_{0}b_{0} = 0  ~~~~~~  \hfill (1)\\
a_{0}b_{1}+a_{1}b_{0} = 0  ~~~~~~  \hfill (2)\\
a_{0}b_{2}+a_{1}b_{1}+a_{2}b_{0} = 0  ~~~~~~  \hfill (3)\\
\ldots ~~~\ldots~~~ \ldots  \\
a_{0}b_{m}+a_{1}b_{m-1}+a_{2}b_{m-2}+\ldots+a_{m}b_{0} = 0 ~~~~~  \hfill (4)\\
a_{0}b_{m+1}+a_{1}b_{m}+a_{2}b_{m-1}+\cdots+a_{m}b_{1} = 0  ~~~~~~  \hfill (5)\\
a_{1}b_{m+1}+a_{2}b_{m}+a_{3}b_{m-1}+\cdots+a_{m}b_{2} = 0  ~~~~~~  \hfill (6)\\
\ldots ~~~~\ldots~~~~\ldots \\
a_{m}b_{n} = 0 ~~~~~~  \hfill (7)\\
\end{array}$\\

Since $R$ is 2-primal and $a_{0}b_{0}\in N_{*}(R)$ by (1), we have $b_{0}a_{0}\in N_{*}(R)$. Now, multiplying by $a_{0}$ from right in (2), we get $a_{0}b_{1}a_{0} \in N_{*}(R)$. This implies $a_{0}b_{1}a_{0}b_{1} \in N_{*}(R)$ and hence, $a_{0}b_{1} \in N_{*}(R)$. Again, multiplying by $a_{1}$ from right in (2), we get $a_{1}b_{0} \in N_{*}(R)$. \\
Now, multiplying (3) by $a_{0}, a_{1}$ and $a_{2}$ respectively from right and using $a_{0}b_{1}, a_{1}b_{0} \in N_{*}(R)$, we get $a_{0}b_{2}, a_{1}b_{1}, a_{2}b_{0} \in N_{*}(R)$. In continuation, multiplying (4) by $a_{0}, a_{1}, a_{2}, \ldots, a_{m}$ and using $a_{i}b_{j} \in N_{*}(R)$, where $1 \leq i+j \leq m-1$, we get $a_{0}b_{m}, a_{1}b_{m-1},\ldots, a_{m}b_{0} \in N_{*}(R)$.  Ultimately, we have $a_{i}b_{j} \in N_{*}(R)$ for each $0\leq i \leq m$, $0\leq j \leq n$. Thus, $R$ is an almost Armendariz ring.
\end{proof}

It has been proved that 2-primal rings are almost Armendariz but converse is not true. In this regard, we have the following example:
\begin{ex} Let $F[X,Y]$ be the free algebra over the field $F$ with noncommuting indeterminates $X, Y$ and $I$ be the ideal $(X^{2})^{2}$ of $F[X,Y]$. Consider the ring $R = F[X, Y]/I$ and $x = X+I$. Then by Example 1 of \cite{Y}, $N(R) = xRx+Rx^{2}R+Fx$ and $N_{*}(R) = Rx^{2}R$. Therefore, $R$ is not 2-primal. Again, if $f(t), g(t) \in R[t]$ such that $f(t)g(t) = 0$, then $ab = 0$ for each $a \in C_{f(t)}$ and $b \in C_{g(t)}$ and hence $R$ is Armendariz ring. Thus, $R$ is an almost Armendariz ring.
\end{ex}

\begin{pro} If $I\subseteq N_{*}(R)$ and $R/I$ is an almost Armendariz, then $R$ is an almost Armendariz.
\end{pro}

\begin{proof} Suppose two polynomials $f(x)$ and $g(x)\in R[x]$ are such that $f(x)g(x) = 0$. Then $\overline{f(x)}\overline{g(x)} = \overline{0}$ in $N_{*}(R/I)[x]$ and this implies that $\overline{a}\overline{b} \in N_{*}(R/I)$, because $R/I$ is an almost Armendariz. We know that $N_{*}(R/I) = N_{*}(R)/I$. Therefore $ab \in N_{*}(R)$ for each $a \in C_{f(x)}$ and $b \in C_{g(x)}$. Thus, $R$ is an almost Armendariz ring.
\end{proof}

\begin{pro} A Semicommutative ring is an almost Armendariz ring.
\end{pro}
\begin{proof}  Let $f(x) = \sum_{i=0}^{m}a_{i}x^{i}$ and $g(x) = \sum_{j=0}^{n}b_{j}x^{j} \in R[x]$ be such that $f(x)g(x) = 0$. Then $\overline{f(x)}\overline{g(x)} = \overline{0} \in R/N_{*}(R)[x]$. Since $R$ is a semicommutative ring, so $R/N_{*}(R)$ is reduced. Hence $(\overline{a_{i}b_{j}}) = \overline{0}\in R/N_{*}(R)$, i.e. $(a_{i}b_{j} + N_{*}(R)) = N_{*}(R),$ for each $i, j$. This implies that $a_{i}b_{j} \in N_{*}(R),$ for each $i, j$ where $0 \leq i \leq m$ and $0 \leq j \leq n$. Hence $R$ is an almost Armendariz ring.
\end{proof}
By above Proposition, Semicommutative rings are almost Armendariz ring but converse is not true. In this regards we have the following example:
\begin{ex}
Let $F$ be a field and $A = F[a, b, c]$ be a free algebra of polynomials with constant term zero in noncommuting indeterminates $a, b, c$ over $F$. Here, $A$ is a ring without identity. Consider an ideal of $F+A=I $ generated by $cc, ac, crc$, for all $r \in A$. Let $R = (F+A)/I$ and denote $a+I$ by $a$. From Example 14 of \cite{C}, $R$ is Armendariz, therefore $R$ is an almost Armendariz ring. But it is not a semicommutative ring because $ac \in I$ but $abc \notin I$.
\end{ex}

\begin{thm} For a ring $R$, let $R/I$ be an almost Armendariz ring for some ideal $I$ of $R$. If $I$ is semicommutative ideal of $R$, then $R$ is an almost Armendariz ring.
\end{thm}
\begin{proof} Let $f(x) = \sum_{i=0}^{m}a_{i}x^{i}$ and $g(x) = \sum_{j=0}^{n}b_{i}x^{j}\in R[x]$ be such that $f(x)g(x) = 0$. This implies $r_{s}f(x)g(x)r_{t} = 0$ for each $r_{s}, r_{t} \in R$ and hence\\
\begin{equation} \label{1}
\sum_{k=0}^{m+n}(\sum_{i+j=k}r_{s}a_{i}b_{j}r_{t})x^{k} = 0
\end{equation}

for each $r_{s}, r_{t} \in R$. Also, $f(x)g(x) = 0$ implies that $\overline{f(x)}\overline{g(x)} = \overline{0} \in (R/I)[x]$. Since $R/I$ is an almost Armendariz ring, so $\overline{a_{i}}\overline{b_{j}} \in N_{*}(R/I)$ for each $i, j$ where $0 \leq i\leq m$ and $0 \leq j \leq n$. This implies there exists a positive integer $n_{ij}$ such that $(r_{s}a_{i}b_{j}r_{t})^{n_{ij}}\in I$ for each $i, j$ and all $r_{s}, r_{t} \in R$. In fact, $n_{ij}$ is maximal among all $n_{i_{s}j_{t}}$ for all $r_{s}$ and $r_{t}$ of $R$.\\
Now, by principle of induction we will prove that $a_{i}b_{j} \in N_{*}(R)$, for each $i$, $j$.\\

If $i+j = 0$, then $a_{0}b_{0} = 0$ implies that $a_{0}b_{0} \in N_{*}(R)$.\\
Let $k$ be a positive integer such that $a_{i}b_{j} \in N_{*}(R)$, where $i+j<k$.\\
Now, we prove $a_{i}b_{j} \in N_{*}(R)$, for $i+j = k$.\\
By hypothesis, $a_{0}b_{k-1} \in N_{*}(R)$ and $N_{*}(R) \subseteq N(R)$. Therefore, there exists a positive integer $u$ such that $(a_{0}b_{k-1})^{u} = 0$, and hence, $(b_{k-1}a_{0})^{u+1} = 0$.\\
For brevity of notation, let $p = n_{0k}$. Then for each pair $r_{s}, r_{t} \in R$,  $(r_{s}a_{0}b_{k}r_{t})^{p} \in I$. Now, we fix $r_1, r_2 \in R$ for further calculation. We have
$$((r_{1}a_{0}b_{k}r_{2})^{p+1}r_{1}a_{1})(b_{k-1}a_{0})^{u+1}(b_{k-1}r_{2}(r_{1}a_{0}b_{k}r_{2})^{p+2}) = 0.$$
Now, $$((r_1a_0b_kr_2)^{p+1}r_1a_1)b_{k-1}, (r_2(r_1a_0b_kr_2)^{p+1}r_1), a_0(b_kr_2(r_1a_0b_kr_2)^{p+1}r_1a_1)\in N(I)$$ and $$(b_{k-1}r_2(r_1a_0b_kr_2)^{p+2}) \in N(I).$$ Therefore, \\
$[((r_{1}a_{0}b_{k}r_{2})^{p+1}r_{1}a_{1})b_{k-1}(r_{2}(r_{1}a_{0}b_{k}r_{2})^{p+1}r_{1})a_{0}(b_{k}r_{2}(r_{1}a_{0}b_{k}r_{2})^{p+1}r_{1}a_{1})(b_{k-1}a_{0})^{u}(b_{k-1}r_{2}(r_{1}a_{0}b_{k}r_{2})^{p+2})] = 0$. This implies that
\begin{eqnarray*}
[((r_{1}a_{0}b_{k}r_{2})^{p+1}(r_{1}a_{1}b_{k-1}r_{2})(r_{1}a_{0}b_{k}r_{2})^{p+2})((r_{1}a_{0}b_{k}r_{2})^{p+1}r_{1}a_{1})(b_{k-1}a_{0})^{u}(b_{k-1}r_{2}(r_{1}a_{0}b_{k}r_{2})^{p+2})] = 0.
\end{eqnarray*}
\\ Continuing this process, we get $$[(r_{1}a_{0}b_{k}r_{2})^{p+1}(r_{1}a_{1}b_{k-1}r_{2})(r_{1}a_{0}b_{k}r_{2})^{p+2}]^{u+2} = 0.$$ This implies $$((r_{1}a_{0}b_{k}r_{2})^{p+1}(r_{1}a_{1}b_{k-1}r_{2})(r_{1}a_{0}b_{k}r_{2})^{p+2})\in N(I).$$ Similarly, we can show that $$((r_{1}a_{0}b_{k}r_{2})^{p+1}(r_{1}a_{i}b_{k-i}r_{2})(r_{1}a_{0}b_{k}r_{2})^{p+2})\in N(I),$$ for $2 \leq i \leq k$.
By equation (\ref{1}), we have
\begin{equation} \label{2}
r_{1}a_{0}b_{k}r_{2}+r_{1}a_{1}b_{k-1}r_{2}+\cdots+r_{1}a_{k}b_{0}r_{2} = 0.
\end{equation}
Multiplying by $(r_{1}a_{0}b_{k}r_{2})^{p+1}$ and $(r_{1}a_{0}b_{k}r_{2})^{p+2}$ from left and right respectively in equation (\ref{2}), we get\\
$$(r_{1}a_{0}b_{k}r_{2})^{p+1}(r_{1}a_{0}b_{k}r_{2})(r_{1}a_{0}b_{k}r_{2})^{p+2} = -[\sum_{i=1}^{k}(r_{1}a_{0}b_{k}r_{2})^{p+1}(r_{1}a_{i}b_{k-i}r_{2})\\(r_{1}a_{0}b_{k}r_{2})^{p+2}] \in N(I).$$ This implies $(r_{1}a_{0}b_{k}r_{2})^{2p+4} \in N(I)$.\\ By same method we can prove that for any pair of $r_{3}, r_{4} \in R$, $(r_{3}a_{0}b_{k}r_{4})^{2p+4} \in N(I)$. Thus, $a_{0}b_{k}\in N_{*}(R)$.\\
Since, $a_{1}b_{k-2} \in N_{*}(R)$ and $N_{*}(R) \subseteq N(R)$. Therefore, there exists a positive integer $v$ such that\\ $(a_{1}b_{k-2})^{v} = 0$, and hence, $(b_{k-2}a_{1})^{v+1} = 0$.\\
Again, let $q = n_{1, k-1}$. Then $(r_{1}a_{1}b_{k-1}r_{2})^{q} \in I$, since $R/I$ is an almost Armendariz ring and hence
$$((r_{1}a_{1}b_{k-1}r_{2})^{q+1}r_{1}a_{2})(b_{k-2}a_{1})^{v+1}(b_{k-2}r_{2}(r_{1}a_{1}b_{k-1}r_{2})^{q+2}) = 0$$

$[((r_{1}a_{1}b_{k-1}r_{2})^{q+1}r_{1}a_{2})b_{k-2}(r_{2}(r_{1}a_{1}b_{k-1}r_{2})^{q+1}r_{1})a_{1}(b_{k-1}r_{2}(r_{1}a_{1}b_{k-1}r_{2})^{q+1}r_{1}a_{2})(b_{k-2}a_{1})^{v}\\(b_{k-2}r_{2}(r_{1}a_{1}b_{k-1}r_{2})^{q+2})] = 0$.\\
Continuing this process, we get $$[(r_{1}a_{1}b_{k-1}r_{2})^{q+1}(r_{1}a_{2}b_{k-2}r_{2})(r_{1}a_{1}b_{k-1}r_{2})^{q+2}]^{v+2} = 0.$$ Therefore,
$$(r_{1}a_{1}b_{k-1}r_{2})^{q+1}(r_{1}a_{2}b_{k-2}r_{2})(r_{1}a_{1}b_{k-1}r_{2})^{q+2} \in N(I).$$ Similarly, $(r_{1}a_{1}b_{k-1}r_{2})^{q+1}(r_{1}a_{i}b_{k-i}r_{2})(r_{1}a_{1}b_{k-1}r_{2})^{q+2} \in N(I)$, for $3 \leq i \leq k$.\\ Suppose, $(r_{1}a_{0}b_{k}r_{2})^{w} = 0$, then $$[(r_{1}a_{1}b_{k-1}r_{2})^{q+1}(r_{1}a_{0}b_{k}r_{2})(r_{1}a_{1}b_{k-1}r_{2})^{q+2}]^{w} = 0.$$\\
Hence, $$[(r_{1}a_{1}b_{k-1}r_{2})^{q+1}(r_{1}a_{0}b_{k}r_{2})(r_{1}a_{1}b_{k-1}r_{2})^{q+2}] \in N(I).$$\\

 Again, multiplying in equation (\ref{2}) by  $(r_{1}a_{1}b_{k-1}r_{2})^{q+1}$, $(r_{1}a_{1}b_{k-1}r_{2})^{q+2}$ from left and right respectively, we get  \newline \newline $(r_{1}a_{1}b_{k-1}r_{2})^{2q+4} = - \sum_{i=2}^{k}[(r_{1}a_{1}b_{k-1}r_{2})^{q+1}(r_{1}a_{i}b_{k-i}r_{2})(r_{1}a_{1}b_{k-1}r_{2})^{q+2}]-$ \newline $$
[(r_{1}a_{1}b_{k-1}r_{2})^{q+1}(r_{1}a_{0}b_{k}r_{2})(r_{1}a_{1}b_{k-1}r_{2})^{q+2}] \in N(I).$$\\
This implies $(r_{1}a_{1}b_{k-1}r_{2})^{2q+4}\in N(I)$, for $r_{1}, r_{2} \in R$.
Therefore, there exists a positive integer $t_2$ such that $[(r_{1}a_{1}b_{k-1}r_{2})^{2q+4}]^{t_{2}} = 0$, for $r_{1}, r_{2}\in R.$ Similarly we can show that for any pair of $r_{3}, r_{4} \in R$, $(r_{3}a_{1}b_{k-2}r_{4})^{2q+4} \in N(I)$. Hence $a_{1}b_{k-1} \in N_{*}(R)$. Continuing the above process, we can show that $a_{2}b_{k-2}, a_{3}b_{k-3}, \ldots a_{k}b_{0} \in N_{*}(R)$. Thus, $a_{i}b_{j} \in N_{*}(R)$, when $i+j = k$.\\
Hence by induction, $a_{i}b_{j} \in N_{*}(R)$ for each $i, j$ where $0 \leq i\leq m$ and $0 \leq j \leq n$. Therefore, $R$ is an almost Armendariz ring.
\end{proof}
\begin{pro} For a ring $R$, let $R/I$ be an almost Armendariz ring for some ideal $I$ of $R$. If $I$ is nilpotent, then $R$ is an almost Armendariz ring.
\end{pro}
\begin{proof} Let $f(x) = \sum_{i=0}^{m}a_{i}x^{i}, g(x) = \sum_{j=0}^{n}b_{j}x^{j}\in R[x]$ such that $f(x)g(x) = 0$. Then $\overline{f(x)} \overline{g(x)}= \overline{0}$. So $(\overline{a_{i}})(\overline{b_{j}}) \in N_{*}(R/I)$, since $R/I$ is an almost Armendariz. Therefore, $(r_{p}a_{i}b_{j}r_{q})^{n_{ij}} \in I$ for each $i, j$ and for any $r_{p}, r_{q}\in R$. Since $I$ is a nilpotent ideal of $R$, $((r_{p}a_{i}b_{j}r_{q})^{n_{ij}})^{k} = 0$, for some $k$, for each $i, j$, where $0 \leq i\leq m$, $0 \leq j\leq n$ and for any $r_{p}, r_{q}\in R$. So $a_{i}b_{j} \in N_{*}(R)$ for each $i, j$. Thus, $R$ is an almost Armendariz ring.
\end{proof}

\section{STRUCTURAL PROPERTIES OF ALMOST ARMENDARIZ RING}
We know by [\cite{D}, Theorem 2] that a ring $R$ is Armendariz if and only if $R[x]$ is an Armendariz ring. But, for weak Armendariz ring $R,$ $R[x]$ need not be weak Armendariz ring. In 2006, Liu and Zhao [\cite{Z}, Theorem ~3.8] proved that if $R$ is semicommutative, then $R[x]$ is weak. Now, for almost Armendariz, we have the following:
\begin{thm} Let $R$ be a ring. Then $R$ is an almost Armendariz ring if and only if $R[x]$ is an almost Armendariz ring.
\end{thm}
\begin{proof}  Let $p(y) = f_{0}(x) + f_{1}(x)y + \cdots + f_{m}(x)y^{m}$, $q(y) = g_{0} + g_{1}(x)y +\cdots +g_{n}y^{n} \in R[x][y]$ such that $p(y)q(y) = 0$, where $f_{i}(x), g_{j}(x) \in R[x]$. Write $f_{i}(x) = a_{i0} + a_{i1}x + \cdots + a_{iu_{i}}x^{u_{i}}$, $g_{j}(x) = b_{j0} + b_{j1}x + \cdots + b_{jv_{j}}x^{v_{j}}$, for each $0 \leq i \leq m$ and $0 \leq j \leq n$, where $a_{i0}, a_{i1}, \ldots, a_{iu_{i}}, b_{j0}, b_{j1}, \ldots, b_{jv_{j}} \in R$. We have to show $f_{i}(x)g_{j}(x) \in P(R[x])$, for each $0 \leq i \leq m$ and $0 \leq j \leq n$. Choose a positive integer $k$ such that $k > deg(f_{0}(x)) + deg(f_{1}(x)) + \cdots + deg(f_{m}(x)) + deg(g_{0}(x)) + deg (g_{1}(x)) + \cdots + deg(g_{n}(x))$. Since $p(y)q(y) = 0 \in R[x][y]$, we get

$$\left \{
\begin{array}{ll}
f_{0}(x)g_{0}(x) = 0\\
f_{0}(x)g_{1}(x) + f_{1}(x)g_{0}(x) = 0\\
\ldots\ldots\ldots\ldots\; \; \; \; \;\; \; \; \; \;\; \; \; \; \; \; \; \; \; \; \; \; \; \; \; \; \; \; \; \; \; \; \; \; \; \; \; \; \; \; \; \; \; \; \;\; \; \; \; \;\; \; \; \; \; (\ast)\\
\ldots\ldots\ldots\ldots\\
f_{m}(x)g_{n}(x) = 0
\end{array}\right.$$
Now put
$$\left \{
\begin{array}{ll}
p(x^{k}) = f(x) = f_{0}(x) + f_{1}(x)x^{k} + f_{2}x^{2k} + \cdots + f_{m}(x)x^{mk};\\
q(x^{k}) = g(x) = g_{0}(x) + g_{1}(x)x^{k} + g_{2}x^{2k} + \cdots + g_{n}x^{nk}. \;\;\;\;\; \; \; \; \; \; \; \; \; \; \; \; \; \; \; \;(\ast\ast)
\end{array}\right.$$
Then
$$f(x)g(x) = f_{0}(x)g_{0}(x) + (f_{0}(x)g_{1}(x) + f_{1}(x)g_{0}(x))x^{k} + \cdots f_{m}(x)g_{n}(x)x^{(n + k)}.$$
Therefore, by $(\ast\ast)$, we have $f(x)g(x) = 0$ in $R[x]$. On the other hand, we have \\ $f(x)g(x) = (a_{00} + a_{01}x + \cdots +a_{0u_{0}}x^{u_{0}} + a_{10}x^{k} + a_{11}x^{k + 1} + \cdots + a_{1u_{1}}x^{k + u_{1}} + \cdots + a_{m0} + a_{m1}x^{mk + 1} + \cdots + a_{mu_{m}}x^{mk + u_{m}})(b_{00} + b_{01}x + b_{0v_{0}}x^{v_{0}} + b_{10}x^{k} + b_{11}x^{k + 1} + \cdots + b_{1v_{1}}x^{k + v_{1}} + \cdots + b_{n0}x^{nk} + b_{n1}x^{nk + 1} + \cdots + b_{nv_{n}}x^{nk + v_{n}}) = 0$.\\  Since $R$ is an almost Armendariz ring, we have $a_{ic}b_{jd} \in P(R),$ for all $0 \leq i \leq m$, $0 \leq j \leq n$, $c \in \{0, 1, \ldots, u_{i}\}$ and $d \in \{0, 1, \ldots, v_{j}\}.$ Therefore, $f_{i}(x)g_{j}(x) \in P(R)[x] = P(R[x]),$ for all $0 \leq i \leq m$ and $0 \leq j \leq n$. Hence, $R[x]$ is an almost Armendariz ring.
\end{proof}
\begin{thm} A ring $R$ is an almost Armendariz ring if and only if $R[x, x^{-1}]$ is an almost Armendariz ring.
\end{thm}
\begin{proof} Let $p(y) = f_{0} + f_{1}y + \cdots + f_{r}y^{r}$, $q(y) = g_{0} + g_{1}(x)y + \cdots + g_{s}y^{s} \in R[x, x^{-1}][y]$ such that $p(y)q(y) = 0$, where $f_{i}'s, g_{j}'s \in R[x,x^{-1}]$. Consider $$f_{i} = a_{i(-n_{i})}x^{-n_{i}} + a_{i(-n_{i}+1)}x^{-n_{i}+1} + \cdots + a_{i(-1)}x^{-1}+a_{i0} + a_{i1}x+ \cdots +a_{im_{i}}x^{m_{i}},$$ $$g_{j} = b_{j(-n'_{j})}x^{-n'_{j}} + b_{j(-n'_{j}+1)}x^{-n'_{j}+1}+ \cdots + b_{j(-1)}x^{-1} + b_{j0} + b_{j1}x+ \cdots + b_{jm'_{j}}x^{m'_{j}},$$ for all $0\leq i \leq r$ and $0\leq j \leq s$, where $a_{i(-n_{i})}, a_{i(-n_{i}+1)}, \ldots, a_{i(-1)}, a_{i0}, a_{i1}, \ldots, a_{im_{i}}$ and \\ $ b_{j(-n'_{j})}, \ldots b_{j(-1)}, b_{j0}, b_{j1},  \ldots b_{jm'_{j}}$ are in $R$. Choose a positive integer $k$ $>$ maximum of $n_{i}$, for $ 0 \leq i \leq r$ and $k'$ $>$  maximum $n'_{j}$, for $0 \leq j \leq s.$ Then $x^{k}p(y), x^{k'}q(y)\in R[x][y]$. Again, consider a positive integer $t$ such that $t > deg(f_{0}) + deg (f_{1})+ \cdots + deg(f_{r}) + deg(g_{0}) + deg(g_{1}) + \cdots + deg(g_{s}),$ then $x^{k}p(x^{t}) = F(x), x^{k'}q(x^{t}) = G(x)\in R[x]$ and hence $F(x)G(x) = 0$. Since $R$ is an almost Armendariz so $a_{ic}b_{jd}\in P(R)$ for each $i, j$ where $0\leq i\leq r$, $0\leq j \leq s$, $c \in \{(-n_{i})\ldots (-1), 0, \ldots, m_{i}\}$ and $d \in \{(-n'_{j})\ldots (-1), 0, \ldots, m'_{j}\}$. This  implies $f_{i}g_{j} \in P(R)[x, x^{-1}] = P(R[x, x^{-1}])$, for each $0\leq i \leq r$, $0\leq j \leq s$. Thus, $R[x, x^{-1}]$ is an almost Armendariz ring.
\end{proof}
Let $R$ be a ring and let $S^{-1}R = \{u^{-1}a ~|~ u \in S, a \in R\}$ with $S$ a multiplicative closed subset of the ring $R$ consisting of central regular elements. Then $S^{-1}R$ is a ring.
\begin{pro} Let $R$ be a ring and $S^{-1}(R)$ as above. If $R$ is an almost Armendariz ring, then $S^{-1}(R)$ is an almost Armendariz ring.
\end{pro}
\begin{proof} Let $R$ be an almost Armendariz ring. Let $F(x) = \sum_{i=0}^{m}\alpha_{i}x^{i}$ and $G(x) = \sum_{j=0}^{n}\beta_{j}x^{j} \in S^{-1}R[x]$ be such that $$F(x)G(x) = 0$$ where $\alpha_{i} = u^{-1}a_{i}$, $\beta_j = v^{-1}b_{j}$ with $a_{i}, b_{j} \in R$ and $u, v \in S$. Now $$F(x)G(x) = (uv)^{-1}(a_{0} + a_{1}x + \cdots +a_{m}x^{m})(b_{0} + b_{1}x + \cdots +b_{n}x^{n}) = 0,$$ $f(x) = \sum_{i=0}^{m}a_{i}x^{i}$ and $g(x) = \sum_{j=0}^{n}b_{j}x^{j} \in R[x]$ with $f(x)g(x) = 0$, since $R$ is an almost Armenadariz ring $a_{i}b_{j} \in N_{*}(R)$. We know that if $G = S^{-1}(R)$, then $N_{*}(G) = S^{-1}(N_{*}(R))$, so $\alpha_{i}\beta_{j} = a_{i}u^{-1}b_{j}v^{-1} \in N_{*}(G)$. Hence $S^{-1}(R)$ is an almost Armendariz.
\end{proof}
Next, we prove that $R$ is an almost Armendariz ring if and only if $U_{n}(R)$ is an almost Armendariz ring. Here, it is noted that\\
\begin{equation*}
N_{*}(U_{n}(R)) = \left(
                    \begin{array}{ccc}
                      N_{*}(R) & R & R \\
                      0 & \ddots & R \\
                      0 & 0 & N_{*}(R) \\
                    \end{array}
                  \right)
\end{equation*}
\begin{pro} A ring $R$ is an almost Armendariz if and only if $U_{n}(R)$ is an almost Armendariz.
\end{pro}
\begin{proof}  Let $R$ be an almost Armendariz ring. Let $f(x) = A_{0}+A_{1}x+A_{2}x^{2}+\cdots+A_{r}x^{r}, g(x) = B_{0}+B_{1}x+B_{2}x^{2}+\cdots+B_{s}x^{s} \in U_{n}(R)[x]$ such that $f(x)g(x) = 0$, where $A_{i}'s$ and $B_{j}'s$ are\\
$A_{i} = \left(
           \begin{array}{cccc}
             a_{11}^{i} & a_{12}^{i} & \ldots & a_{1n}^{i} \\
             0 & a_{22}^{i} & \ldots & a_{nn}^{i} \\
             \vdots & \vdots & \ddots & \vdots \\
             0 & 0 & \ldots & a_{nn}^{i} \\
           \end{array}
         \right)$,~~~~~~~~~~~~~~~~
         $B_{j} = \left(
           \begin{array}{cccc}
             b_{11}^{j} & b_{12}^{j} & \ldots & b_{1n}^{j} \\
             0 & b_{22}^{j} & \ldots & b_{nn}^{i} \\
             \vdots & \vdots & \ddots & \vdots \\
             0 & 0 & \ldots & b_{nn}^{j} \\
           \end{array}
         \right)$.\\

Then from $f(x)g(x) = 0$, we have $\Big(\sum _{i=0}^{r}a_{pp}^{i}x^{i}\Big)\Big(\sum_{j=0}^{s}b_{pp}^{j}x^{j}\Big) = 0\in R[x]$, for $p = 1, 2\ldots n$. Since $R$ is an almost Armendariz, $a_{pp}^{i}b_{pp}^{j} \in N_{*}(R)$, for each $p$ and each $i$, $j$. Therefore, $A_{i}B_{j} \in N_{*}(U_{n}(R))$ for each $i, j$. Hence, $U_{n}(R)$ is an almost Armendariz ring.
\end{proof}
\begin{cor} If $R$ is an Armendariz ring, then for a positive integer $n$, $U_{n}(R)$ is an almost Armendariz ring.
\end{cor}

\begin{pro} Let $R$ be a reduced ring. Then $D_{n}(R)$ is an almost Armendariz.
\end{pro}
\begin{proof} Reduced rings are Armendariz rings, therefore by Corollary $(3.1)$, $D_{n}(R)$ is an almost Armendariz.
\end{proof}
\begin{df} Given a ring $R$ and a bimodule $_{R}M_{R}$, the trivial extension of $R$ by $M$ is the ring $T(R, M)$ with the usual addition and multiplication defined as
$$(r_{1}, m_{1})(r_{2}, m_{2}) = (r_{1}r_{2}, r_{1}m_{2}+m_{1}r_{2}).$$
This is isomorphic to the ring of all matrices of the form $\left(
                                                 \begin{array}{cc}
                                                   r & m \\
                                                   0 & r \\
                                                 \end{array}
                                               \right)$ with usual addition and multiplication of  matrices, where $r\in R$ and $m\in M$.
\end{df}
\begin{cor} Let $R$ be a reduced ring. Then trivial extension $T(R, R)$ is an almost Armendariz.
\end{cor}

Towards the property of an almost Armendariz ring for $M_{n}(R)$, we have the following :
\begin{ex} Let $F$ be a field and $R=M_{n}(R)$ . Let $f(x) = e_{11}x -e_{12}x$ and $g(x) = e_{21}+e_{11}x$, then $f(x)g(x) = 0$. But $e_{11}e_{11} = e_{11}$ is not strongly nilpotent. Thus, $R$ is not an almost Armendariz.
\end{ex}

\begin{df} For an algebra $R$ over commutative ring $S$, the Dorroh extension of $R$ by $S$ is an abelian group $D = R\oplus S$ with multiplication given by $(r_1, s_1)(r_2, s_2) = (r_1r_2+s_1r_2+s_2r_1, s_1s_2)$, where $r_1, r_2 \in R$ and $s_1, s_2 \in S$.
\end{df}
With this we have the following :
\begin{thm} Let $R$ be an algebra over a commutative domain $S$ and $D$ the Dorroh extension of $R$ by $S$. Then $R$ is an almost Armendariz ring if and only if $D$ is an almost Armendariz ring.
\end{thm}
\begin{proof} We notice that, $s.1 \in R$ for any $s \in S$. So $R = \{r+s : (r, s) \in D\}$. Therefore $N_{*}(D) = N_{*}(R) \oplus \{0\}$.\\
Let $D$ be an almost Armendariz ring. Since $D$ is trivial extension of $R$ by $S$. Therefore, $R$ is an almost Armendariz ring.\\
 Conversely, let $R$ be an almost Armendariz ring. Let $f(x) = \sum_{i=0}^{m}(a_{i}b_{i})x^{i} = (f_{1}(x), f_{2}(x))$ and $g(x) = \sum(c_{j}, d_{j})x^{j} = (g_{1}(x), g_{2}(x)) \in D[x]$ be such that $f(x)g(x) = 0$ where,\\ $f_{1}(x) = \sum_{i=0}^{m}a_{i}x^{i},$ $f_{2}(x) = \sum_{i=0}^{m}b_{i}x^{i},$ $g_{1}(x) = \sum_{j=0}^{n}c_{j}x^{j}$ and $g_{2}(x) = \sum_{j=0}^{n}d_{j}x^{j}$. From $f(x)g(x) = 0$, we have $$f_{1}(x)g_{1}(x)+f_{1}(x)g_{2}(x)+f_{2}(x)g_{1}(x) = 0$$ and $$f_{2}(x)g_{2}(x) = 0.$$ Since $S$ is a domain, therefore, either $f_{2}(x) = 0$ or $g_{2}(x) = 0$. \\

\textbf{Case 1.} If $f_{2}(x) = 0$, then from $$f_{1}(x)g_{1}(x)+f_{1}(x)g_{2}(x)+f_{2}(x)g_{1}(x) = 0,$$ we have $f_{1}(x)g_{1}(x)+f_{1}(x)g_{2}(x) = 0$ and this implies $f_{1}(x)(g_{1}(x)+g_{2}(x)) = 0$. Since $R$ is an almost Armendariz ring, we have $a_{i}(c_{j}+d_{j}) \in N_{*}(R)$ . Hence $(a_{i}, 0)(c_{j}, d_{j}) = (a_{i}c_{j}+a_{i}d_{j}, 0) \in N_{*}(R) \oplus \{0\}$ for each $i, j$. Thus, $D$ is an almost Armendariz ring.\\

\textbf{Case 2.} If $g_{2}(x) = 0$, then from $$f_{1}(x)g_{1}(x)+f_{1}(x)g_{2}(x)+f_{2}(x)g_{1}(x) = 0,$$ we have $(f_{1}(x)+f_{2}(x))g_{1}(x) = 0$. Since $R$ is an almost Armendariz ring, therefore, $(a_{i}+b_{i})c_{j} \in N_{*}(R)$ . Hence $(a_{i}, b_{i})(c_{j}, 0) = (a_{i}c_{j}+b_{i}c_{j}, 0) \in N_{*}(R)\oplus \{0\}$, for each $i, j$. Thus, $D$ is an almost Armendariz ring.
\end{proof}
\begin{pro} Let $R_{i}$ be rings for $i \in I$. Then $\prod R_{i}$ $(\oplus R_{i})$ is an almost Armendariz ring if and only if $R_{i}$ is an almost Armendariz ring for each $i \in I$.
\end{pro}
\begin{proof} We have, $N_{*}(\prod_{i \in I} R_{i}) = \prod_{i \in I} N_{*}(R_{i})$ and $N_{*}(\oplus_{i \in I} R_{i}) = \oplus_{i \in I} N_{*}(R_{i})$. Let $f(x), g(x) \in \prod _{i \in I} R_{i}[x]$ be such that $f(x)g(x) = 0$. We also have, $(\prod _{i \in I}R_{i})[x] = \prod _{i \in I}R_{i}[x]$, therefore $f(x) = \prod _{i \in I}(f_{i}(x)) \in \prod _{i \in I}R_{i}[x]$ and $g(x) = \prod _{i \in I}(g_{i}(x)) \in \prod _{i \in I}R_{i}[x]$. We know that each $R_{i}$ is an almost Armendariz ring, so $\prod _{i \in I} R_{i}$ is Armendariz ring. Similarly we can show that $\oplus _{i \in I}R_{i}$ is an almost Armendariz ring.
\end{proof}

Finally, we calculated the minimal order of a noncommutative almost Armendariz ring. In \cite{K}, Eldridge proved some results for the order of finite noncommutative ring with unity as follows:\\
\begin{itemize}
\item [(1)] A finite ring $R$ with identity is commutative if $|R|$ has cube free factorization.
\item [(2)]  If $R$ is noncommutative ring with identity and $|R| = p^{3}$, then $R$ is isomorphic to $\left(
                                                                    \begin{array}{cc}
                                                                      GF(p) & GF(p) \\
                                                                      0 & GF(p) \\
                                                                    \end{array}
                                                                  \right)$.
\end{itemize}

\begin{ex}

Let $R = \left\{\left(
             \begin{array}{cc}
               a & b \\
               0 & c \\
             \end{array}
           \right): a, b, c \in \mathbb{Z}_2\right\}$. Then by Proposition (3.2), $R$ is an almost Armendariz ring.
\end{ex}
\begin{pro}
Let $R$ be a noncommutative ring with identity. Then minimal cardinality of $R$ is $8$ to be an almost Armendariz ring.\\
 \end{pro}
 \begin{proof}
If $R$ has cardinality 8, then it is isomorphic to $\left(
                                                                                         \begin{array}{cc}
                                                                                           GF(2) & GF(2) \\
                                                                                           0 & GF(2) \\
                                                                                         \end{array}
                                                                                       \right)$.
 Hence from Example (8), $R$ is almost Armendariz.
 \end{proof}
\begin{ex} Let $D$ be a domain and $R = \left(
                                        \begin{array}{cc}
                                          D & D \\
                                          0 & 0 \\
                                        \end{array}
                                      \right)$, $S = \left(
                                                       \begin{array}{cc}
                                                         0 & D \\
                                                         0 & D \\
                                                       \end{array}
                                                     \right)$ be two rings.
                                      Let $0 \neq f(x), 0 \neq g(x) \in R[x]$ be such that $f(x)g(x) = 0$. Now, we can express $f(x)$, $g(x)$ as follows:\\
                                      $ f(x) = \left(
                                                 \begin{array}{cc}
                                                   f_{0} & f_{1} \\
                                                   0 & 0 \\
                                                 \end{array}
                                               \right)$, $g(x) = \left(
                                                                   \begin{array}{cc}
                                                                     g_{0} & g_{1} \\
                                                                     0 & 0 \\
                                                                   \end{array}
                                                                 \right)$ where $f_{0}, f_{1}, g_{0}, g_{1} \in D[x]$. From $f(x)g(x) = 0$, we get $f_{0}g_{0} = 0$ and $f_{0}g_{1} = 0$. If $ f_{0} = 0$, then $ab \in N_{*}(R)$ for each $a \in C_{f}$ and $b \in C_{g}$. If $f_{0} \neq 0$, then $g_{0} = 0$, $g_{1} = 0$ and $ab \in N_{*}(R)$ for each $a \in C_{f}$ and $b \in C_{g}$. Thus, $R$ is an almost Armendariz ring. Similarly, we can prove that $S$ is an almost Armendariz ring.
\end{ex}
\begin{thm} Let $R$ be a non-commutative ring without identity. Then the minimal cardinality of an almost Armendariz ring is 4.
\end{thm}
\begin{proof} Let $R$ be a minimal noncommutative almost Armendariz ring without identity. Since $R$ is commutative when $|R|\leq 3$, therefore, $|R| \geq 4$. If $|R| = 4$ and $R$ is nil, then $R$ is nilpotent as well as commutative by Lemma (2.7) of \cite{TT}, which contradicts the assumption. Therefore, $R$ must be non-nil and $|J(R)| = 0$ or $|J(R)| = 2$. If $|J(R)| = 0$, then $R$ is also commutative by Theorem 3.4 of \cite{YY}, again contradicts the assumption. Hence, $J(R)$ must have cardinality 2. Therefore, by Theorem (3.4) of \cite{YY}, $R$ is isomorphic to $\left(
                \begin{array}{cc}
                  \mathbb{Z}_2 & \mathbb{Z}_2 \\
                  0 & 0 \\
                \end{array}
              \right)$ or $\left(
                            \begin{array}{cc}
                              0 & \mathbb{Z}_2 \\
                              0 & \mathbb{Z}_2 \\
                            \end{array}
                          \right)$.

\end{proof}



\end{document}